\newcommand{\bc}{\begin{center}}
\newcommand{\ec}{\end{center}}
\newcommand{\benum}{\begin{enumerate}}
\newcommand{\eenum}{\end{enumerate}}

\newcommand{\matl}{\left[ \begin{array}}
\newcommand{\matr}{\end{array} \right]}
\newcommand{\matls}{\left[ \begin{smallmatrix}}
\newcommand{\matrs}{\end{smallmatrix} \right]}

\newcommand{\isdef}{\stackrel{\triangle}{=}}

\newcommand{\rmT}{{\rm T}}

\newcommand{\BBR}{{\mathbb R}}

\documentclass[letterpaper,12pt,draftcls,onecolumn]{IEEEtran}

\IEEEoverridecommandlockouts                              

\usepackage{amsmath}
\usepackage{amsfonts}
\usepackage{mathrsfs}
\usepackage{amssymb,latexsym}
\usepackage[psamsfonts]{eucal}
\usepackage{amsthm}
\usepackage{mathtools}
\usepackage{graphicx}
\usepackage{epsfig}
\usepackage{psfrag}
\usepackage{pstricks}
\usepackage{amssymb}
\usepackage{amsmath}
\usepackage{setspace}
\usepackage{ulem} 
\usepackage{color}
\usepackage{resizegather}
\newtheorem{theorem}{\indent Theorem}[section]
\newenvironment{thm}{\begin{theorem}$\!\!${\bf }\rm }{\end{theorem}}
\newtheorem{proposition}{\indent Proposition}[section]
\newenvironment{prop}{\begin{proposition}$\!\!${\bf }\rm }{\end{proposition}}
\newtheorem{lemma}{\indent Lemma}[section]
\newenvironment{lem}{\begin{lemma}$\!\!${\bf }\rm }{\end{lemma}}
\newtheorem{corollary}{\indent Corollary}[section]
\newenvironment{cor}{\begin{corollary}$\!\!${\bf }\rm }{\end{corollary}}
\newtheorem{definition}{\indent Definition}[section]
\newenvironment{defn}{\begin{definition}$\!\!${\bf }\rm }{\end{definition}}
\newtheorem{example}{\indent Example}[section]

\newtheorem{Fact}{\indent Fact}[section]
\newenvironment{fact}{\begin{Fact}$\!\!${\bf }\rm }{\end{Fact}}
\newtheorem{conjecture}{\indent Conjecture}[section]
\newenvironment{con}{\begin{conjecture}$\!\!${\bf }\rm }{\end{conjecture}}

\makeatletter \@addtoreset{equation}{section}
\def\theequation{\arabic{section}.\arabic{equation}}
\def\thesection{\arabic{section}}
\makeatother
\title{Delayed Recursive State and Input Reconstruction\thanks{*
This work was supported by Department of Science and Technology through
projects SR/S3/MERC/0064/2012 and SERC/ET-0150/2012, and in part by IIT
Gandhinagar.}}

\author{Roshan A Chavan$^\dag$ and Harish J.
  Palanthandalam-Madapusi$^\ddag$\\ SysIDEA Lab\\ IIT Gandhinagar,\\
  Ahmedabad - 382424, India.\thanks{$\dag$ Graduate Student, Mechanical
    Engineering, System Identification, Data based Estimation and Analysis Lab, IIT
    Gandhinagar, Ahmedabad - 382424, India,
    roshan\_chavan@iitgn.ac.in. \par
    $\ddag$ Assistant Professor, Mechanical Engineering, System
    Identification, Data based Estimation and Analysis Lab, IIT Gandhinagar, Ahmedabad -
    382424, India, +91 79 3245 9899, harish@iitgn.ac.in.}}

\begin{document}
\maketitle \doublespacing
\begin{abstract}
  The unknown inputs in a dynamical system may represent unknown
  external drivers, input uncertainty, state uncertainty, or
  instrument faults and thus unknown-input reconstruction has several
  wide-spread applications.  In this paper we consider delayed
  recursive reconstruction of states and unknown inputs for both
  square and non-square systems. That is, we develop filters that use
  $current$ measurements to estimate $past$ states and reconstruct $past$
  inputs. We further derive necessary and sufficient conditions for
  convergence of filter estimates and show that these convergence
  properties are related to multivariable zeros of the system. With
  the help of illustrative examples we highlight the key contributions
  of this paper in relation with the existing literature. Finally, we
  also show that existing unbiased minimum-variance filters are
  special cases of the proposed filters and as a consequence the
  convergence results in this paper also apply to existing unbiased
  minimum-variance filters.
\end{abstract}

\section{Introduction}



Unknown inputs in a dynamical system may represent unknown external
drivers, input uncertainty, state uncertainty, or instrument
faults. Thus both reconstruction of unknown-inputs and estimation of states in the presence of unknown inputs, have numerous applications in
all fields of engineering. These are fundamental problems that have been of interest for the last several decades with a range of papers relating to state estimation and unknown-input reconstruction\cite{sain69,silver69,dorato69,moylan77,bhattacharyya78,kudva1980,miller82,fairman84,kitanidis,yang1988,darouach94,corless98,hou1998,shabaik03,kitanidis,sundaram,gmarroz2010,kirtikar2010}. While, both discrete-time and continuous-time versions of the problem received attention, in the discrete-time setting, the problem can be stated in its simplest form as the problem of estimating the state $x_k$ and/or the unknown inputs $e_k$ for linear systems of the form
\begin{align}
       x_{k+1} &= Ax_k+He_k, \label{1}\\
       y_k &= Cx_k \label{2},
\end{align}
using knowledge of the model equations and measurements of the outputs $y_k$ alone.

The early works in this area \cite{sain69,silver69,dorato69} approached this as a system inversion problem and focussed on observability conditions under which estimation of unknown inputs are possible. Subsequently, a number of papers over the next several years focussed on construction of observers for state estimation in the presence of unknown inputs \cite{moylan77,bhattacharyya78,kudva1980,miller82,fairman84,kitanidis,yang1988,darouach94,hou1998,shabaik03} with varying approaches.

More recently, interest has turned to reconstructing the unknown inputs in addition to estimation of the states \cite{corless98,xiong03,floquet2006,gmarroz2010,gmarrocdc2010,steven_kitanidis,palanthACC2006,palanthUMVACC2007}. Some work like \cite{kitanidis,palanthUMVACC2007} suggest that input reconstruction can be conceived as an added step after unbiased estimates of the states of the systems are obtained. However, both the state estimation literature and input reconstruction literature focussed on estimating the states or inputs at the immediate previous time step given output measurements until the current time step. Such an approach invariably led to an assumption that $CH$ has full column rank or a closely related assumption. This assumption ensured that all the unknown inputs at time step $k-1$ directly affected the outputs at time step $k$ (as is obvious from a  simple substitution in (\ref{1}) and (\ref{2})), and therefore was a necessary condition for being able to estimate $x_k$ and/or $e_{k-1}$ from $y_{k}$. This becomes a fairly restrictive assumption as there are large classes of systems in which the effect of all the unknown inputs may not be seen in the output in the immediate next time step but may be seen in subsequent time steps (when $CH$ does not have full column rank). Furthermore, convergence results for the filters developed re largely missing the literature.

Recent work on input and state observability \cite{kirtikar2010} question the need for this assumption and in fact conclude that in the case that $CH$ is not full column rank but other conditions are satisfied, it may be possible estimate inputs and states with a time shift (with a delay). That is, it may be possible to estimate $e_{k-r}$ or $x_{k-r}$ given measurements of $y_k$. However, \cite{kirtikar2010} does not provide a robust, recursive way to estimate these states and inputs. \cite{sundaram,amatoACC2013,palanthfitch,gmarroz2010} take advantage of this idea to explore recursive filter-based methods to estimate past (delayed) states and inputs based on measurements of current outputs. \cite{palanthfitch} represent some initial preliminary efforts in this direction, while \cite{amatoACC2013} develops a heuristic method for square systems (dimension of inputs are same as dimension of outputs). \cite{gmarroz2010} incorporates a reconstruction delay with the purpose of negating the effect of non-minimum-phase zeros on the reconstruction error. Note that it has been established in \cite{kirtikar2010} and other related works that non-minimum-phase invariant zeros in the system present a fundamental limitation in reconstruction of unknown inputs and states.

\cite{sundaram} develops a filter that uses a bank of measurements from current time ($k$) to a past time ($k-r$) to estimate the states at time instance $k-r$. While this paper is able to successfully relax the assumption that $CH$ must have full column rank and drawing connections with presence of invariant zeros in the system, it does not focus on input reconstruction and focusses on state estimation alone. Further, the results are not connected to observability results present int he literature.

\par
In this paper, we develop a novel but relatively simple class of filters that incorporates
a reconstruction delay in estimating the states and the unknown inputs of the system. This reconstruction delay allows us to relax the assumption of $CH$ having full column rank and thus are applicable to a larger class of systems. These
filters use measurements up to time step $k$ to reconstruct states and
inputs up to time step $k-r$, where $r$ is a non-negative integer (see
Figure \ref{fig:dir}). This is referred to as reconstruction of the
inputs with a delay of $r$ time steps and should not be confused with
the system dynamics having a delay. We further show that the reconstruction delay $r$ is not the choice of the user but rather dependent on the system and can be characterised in terms of the rank of matrices containing markov parameters.


For these new filters, we then
develop necessary and sufficient conditions for convergence and
investigate their relationship with the multivariable invariant zeros of the
system. We further establish conditions under which a stable filter exists and relate these conditions to system invertibility results.


\par
Thus we make three important contributions. First, we develop a
new class of filters that apply to a wide range of systems by incorporating
an appropriate delay in input reconstruction, with no restrictions on the nature
of input. Second, we develop sufficient conditions for convergence of the estimates
that is largely absent in the literature. By further establishing that several existing filters are special cases of the current filter, we are also indirectly proving convergence results for several filters int he literature. Third, we establish a connection between invertibility results in the literature with conditions for existence and convergence of the proposed filters. Since invertibility and observability literature and filtering literature have been mostly disparate thus far, this is a contribution that helps connect two sets of results in the literature.

\begin{figure}[ht!]
\centering
\includegraphics[width=8cm]{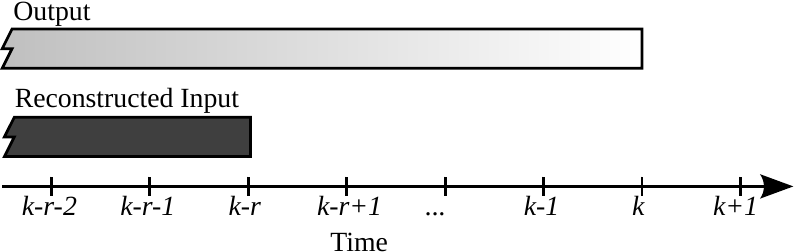}
\caption{Concept of input reconstruction with a delay.}
\label{fig:dir}
\end{figure}

We first start by introducing and developing the new filter in the following section.

\section{Unbiased Filter with General Delay}

Consider state estimation and input reconstruction with a
delay of $r$ time steps, that is, measurements up to time
step $k$ are used to estimate states and inputs at time step $k-r$.
Consider a state-space system
\begin{align}
       x_{k+1}= Ax_k+Bu_k+He_k+w_k, \label{x}
\end{align}
\begin{align} 
     y_k = Cx_k + Du_k+v_k, \label{y}
\end{align}
$x_k \in \BBR^n,\, u_k \in \BBR^m,\, y_k\in \BBR^l,\, e_k \in \BBR^p$, are the
state, known input, output measurement, unknown input vectors, respectively,
$w_k\in \BBR^n$ and $v_k \in \BBR^l$ are zero-mean, white process and
measurement noise,
respectively, and $A \in \BBR^{n \times n},\, B \in \BBR^{n \times m},\, H \in
\BBR^{n \times p},\, C \in \BBR^{l \times n},\,$ and $ D \in \BBR^{l \times
m}\,$. Note that $e_k$ is an arbitary unknown input and can represent either
deterministic or stochastic unknown signals. 
First we consider the simplifications $B = 0$ and $D = 0$. Note that the filter
derivation is independent of $B$ and $D$ matrices, and thus the assumption on
$B$ and $D$ matrices is for convenience alone. 
Without loss of generality, we assume $l \leq n$ also we assume $p<n$
and rank$(H)=p$. $l > n$ would imply the presence of redundant sensors.

For the state-space system (\ref{x}), (\ref{y}) (and with $B=0$ and $D=0$), we
consider a filter of the form
\begin{align}
\hat{x}_{k-r\mid k} = \hat{x}_{k-r\mid k-1}+L_{k}(y_{k}-C\hat{x}_{k\mid
k-1}), \label{xhat1r}
\end{align}
where
\begin{align}
     \hat{x}_{k-r\mid k-1} = A\hat{x}_{k-r-1\mid k-1},\;\hat{x}_{k\mid k-1} = A^{r+1}\hat{x}_{k-r-1\mid k-1}.\label{xhat3r} 
\end{align}

The unique feature of the above filter equations is that estimates are computed
with a delay of $r$ time steps. That is, $\hat{x}_{k-r\mid k}$ is the state
estimate at time step $k-r$ given output data (measurements) up to the
current time step $k$. Note that $\hat{x}_{k\mid k-1}$ is a $r+1$ step open
loop prediction based on the previous state estimate $\hat{x}_{k-r-1\mid k-1}$

Next, we define the state estimation error as
\begin{align}
     \varepsilon_{k-r} \triangleq x_{k-r} - \hat{x}_{k-r\mid k}, \label{see_r}
\end{align}
and the error covariance matrix as 
\begin{align}
P_{k-r\mid k} \triangleq \mathbb{E}[\varepsilon_{k-r}\varepsilon_{k-r}^T].
\label{ecmg}
\end{align}

\subsection{Necessary Conditions for Unbiasedness}
\begin{defn} \label{defn:unb}
The filter (\ref{xhat1r}) - (\ref{xhat3r}) is \emph{unbiased} if $\hat{x}_{k-r
\mid k}$ is an unbiased estimate of the state $x_{k-r}$.
\end{defn}
Definition \ref{defn:unb} implies that the filter (\ref{xhat1r})-(\ref{xhat3r})
is unbiased if and only if
\begin{align}
    \mathbb{E}[\varepsilon_{k-r}]=\mathbb{E}[x_{k-r} - \hat{x}_{k-r\mid k}]=0.
\label{ev1}
\end{align}
Next, we note that
\begin{align}
\varepsilon_{k-r}  = & x_{k-r} - \hat{x}_{k-r\mid k} \nonumber \\
 =  & A x_{k-r-1} + H e_{k-r-1} + w_{k-r-1} - \hat{x}_{k-r\mid k-1}  -
L_{k}(Cx_k+v_k - \nonumber \\ & C \hat{x}_{k\mid k-1}).
\nonumber \\
= &(A-L_{k}CA^{r+1})\varepsilon_{k-r-1}+ (H-L_{k}CA^rH)e_{k-r-1} \nonumber\\
& -L_{k}CA^{r-1}He_{k-r}-L_{k}CA^{r-2}He_{k-r+1}+...+\nonumber\\ &
L_{k}CAHe_{k-2} -L_{k}CHe_{k-1}+
w_{k-r-1} -L_{k}(CA^rw_{k-r-1}+\nonumber\\ &
CA^{r-1}w_{k-r}+...+Cw_{k-1}+v_{k}).\label{eq:error_r}
\end{align}
\begin{thm} \label{thm:nec}
Let $L_k$ be such that the filter (\ref{xhat1r}) - (\ref{xhat3r}) is unbiased.
Then
\begin{align}
H-L_{k}CA^rH =&L_{k}CA^{r-1}H=L_{k}CA^{r-2}H=\notag \\ & ...=
L_{k}CH = 0 \label{cond_r}
\end{align}
\end{thm}
{\bf Proof:}
Since by definition, filter (\ref{xhat1r}) - (\ref{xhat3r}) is unbiased if and
only if $ \mathbb{E}[\varepsilon_{k-r-1}]=0$, it follows from (\ref{eq:error_r})
that
\begin{align}
\mathbb{E}[\varepsilon_{k-r}] =& \mathbb{E}
[(A-L_{k}CA^{r+1})\varepsilon_{k-r-1}+
(H-L_{k}CA^rH)e_{k-r-1} \nonumber\\
& -L_{k}CA^{r-1}He_{k-r}-L_{k}CA^{r-2}He_{k-r+1}+...\nonumber\\
& +L_{k}CAHe_{k-2} -L_{k}CHe_{k-1}+w_{k-r-1}-\nonumber\\
&L_{k}(CA^rw_{k-r-1}+CA^{r-1}w_{k-r}+...+Cw_{k-1}+v_{k})] \label{ev3_r}
\end{align}
Since (\ref{ev3_r}) must hold for arbitrary input sequence $e_k$, it follows
that (\ref{cond_r}) must hold for filter (\ref{xhat1r}) - (\ref{xhat3r}) to be
unbiased. \qed
\begin{cor} \label{cor_r}
Let $L_k$ be such that the filter (\ref{xhat1r}) - (\ref{xhat3r}) is unbiased.
Then, the following conditions hold
\begin{enumerate}
\item[$i)$] $p \leq l$,
\item[$ii)$] rank$(CA^rH)=p$,
\item[$iii)$] rank$(L_{k}) \geq p,\,$ for all $k$,
\item[$iv)$] rank$(CA^dH)\leq l-p,$ for $d=0,\dots,r-1$.
\end{enumerate}
\end{cor}
{\bf Proof.} Since the filter (\ref{xhat1r})- (\ref{xhat3r}) is unbiased, it
follows from Theorem \ref{thm:nec} that (\ref{cond_r}) holds and hence
\begin{align}
 L_{k}(CA^rH)=H.\label{eq1}
\end{align}
Since rank($H$)=$p$, it then follows from (\ref{eq1}) that $iii)$ holds and
\begin{align}
 \text{rank}(CA^rH)\geq p \label{eq2}
\end{align}
Since $CA^rH\in \mathbb{R}^{l\times p}$, it follows from (\ref{eq2}) that
statement $i)$ holds. Furthermore, it follows from (\ref{eq2}) and $i)$ that
statement $ii)$ holds.
\par Finally to prove $iv)$, since (\ref{cond_r}) holds, it follows from
\cite[Proposition 2.5.9, p. 106]{bernsteinbook} that rank$(L_kCA^dH)=0$ and
therefore
\begin{align}
\mbox{rank}(L_k) + \mbox{rank}(CA^dH) & \leq \mbox{rank}(L_{k}CA^dH) + l
\nonumber \\
& = l. \label{eq:rankch1g}
\end{align}
Furthermore, using $iii)$, (\ref{eq:rankch1g}) becomes
\begin{align}
 p+\text{rank}(CA^dH)\leq l,
\end{align}
that is, rank$(CA^dH)\leq l-p.$ \qed

\begin{cor}\label{sq1}
Let $l=p$ and let $L_k$ be such that the filter (\ref{xhat1r}) - (\ref{xhat3r})
is unbiased, and let $l=p$. Then, $CA^dH = 0$ for $d=0,\,\hdots,\,r-1$ and
rank$(L_{k})=p$ for all $k$ and $r$.
\end{cor}
Next, we define the Rosenbrock matrix $Z(s)$ as
\begin{align}
Z(s) \triangleq \left[ \begin{array}{cc}
 (s I-A) & H  \\
 C & 0
\end{array}  \right]. \nonumber
\end{align}
For an unbiased filter, we know from Corollary \ref{cor_r}, i)
that $p\leq l$, therefore the transfer function $G(s) \isdef
C(sI-A)^{-1}H$ has normal rank equal to $p$.

\begin{defn} \label{defn:inv_zr}
$z \in \mathbb{C}$ is an \emph{invariant zero}  \cite{bernsteinbook} of 
(\ref{x}), (\ref{y}) if
\begin{align}
{\rm rank\:}Z(z) <{\rm normal \: rank \:}Z(s). \nonumber
\end{align}
\end{defn}
\begin{defn} \label{defn:asymp_unb}
The filter (\ref{xhat1r})-(\ref{xhat3r}) is \emph{asymptotically unbiased} if
\begin{align}
\lim_{k \to \infty} \mathbb{E}[\varepsilon_{k-r}] = 0 \nonumber
\end{align}
\end{defn}
Definition \ref{defn:asymp_unb} implies that the filter (\ref{xhat1r}) -
(\ref{xhat3r}) is asymptotically unbiased if and only if
$\hat{x}_{k-r\mid k}$ converges to an unbiased estimate of the state $x_{k-r}$
as $k$ approaches infinity.\\
In the following subsection we first develop the filter and examine its
convergence properties for a square system $(l=p)$, and treat the non-square
case later.

\subsection{Square Systems}
\subsubsection{Sufficient Conditions for Unbiasedness}

\begin{lemma} \label{lem:L_r}
Let $l=p$ and let $L_{k}$ be such that (\ref{cond_r}) holds. Then
\begin{align}
L_{k} = H(CA^rH)^{-1}. \label{eq:L_squarer}
\end{align}
\end{lemma}

{\bf Proof.}
The proof follows from (\ref{cond_r}) and ($ii$) of Corollary \ref{cor_r} \qed

\begin{lemma} \label{lem:eig_r}
Let $l=p$ and let $L_{k}$ be such that (\ref{cond_r}) hold. Then all non-zero
eigenvalues of $(A-L_{k}CA^{r+1})$ are invariant zeros of (\ref{x}) and
(\ref{y}).
\end{lemma}

{\bf Proof.}
Let $ \lambda \neq 0$ be an eigenvalue of $(A-L_{k}CA^{r+1})$. Then using
Lemma
\ref{lem:L_r}, it follows that
\begin{align}
{ \rm det}(A-H(CA^rH)^{-1}CA^{r+1} - \lambda I) = 0. \nonumber
\end{align}
Next, let $\nu \neq 0$ be such that
\begin{align}
(A-H(CA^rH)^{-1}CA^{r+1} - \lambda I) \nu =0, \nonumber
\end{align}
and thus
\begin{align}
(A- \lambda I) \nu - H(CA^rH)^{-1}CA^{r+1} \nu = 0. \nonumber
\end{align}
Next, defining 
\begin{align}
\mu \triangleq - (CA^rH)^{-1}CA^{r+1} \nu,\label{sqcom}
\end{align}
it follows that
\begin{align}
(A- \lambda I) \nu + H\mu = 0,\label{sqim1}
\end{align}
and thus yielding 
\begin{align}
 A\nu=\lambda \nu-H\mu\label{p1}
\end{align}
and
\begin{align}
 A^{r+1}\nu=\lambda A^r\nu-A^rH\mu. \label{p2}
\end{align}
Repeatedly using (\ref{p1}) in (\ref{p2}), we get 
\begin{align}
A^{r+1}\nu=\lambda^{r+1}\nu-\lambda^rH\mu-\lambda^{r-1}AH\mu-\hdots-\lambda
A^{r-1}H\mu-A^rH\mu. \nonumber
\end{align}
Left multiplying by C on both sides gives
\begin{align}
 CA^{r+1}\nu=\lambda^{r+1}C\nu-\lambda^rCH\mu-\lambda^{r-1}CAH\mu-\hdots-\lambda
CA^{r-1}H\mu-CA^rH\mu. \nonumber
\end{align}
Since from Corollary \ref{sq1} it follows that $CA^dH=0$ for
$d=0,\,\hdots,\,r-1$ and since $CA^rH$ is invertible, we have
\begin{align}
(CA^rH)^{-1}CA^{r+1}\nu=\lambda^{r+1}(CA^rH)^{-1}C\nu-\mu,\nonumber
\end{align}
and thus
\begin{align}
 \mu=-(CA^rH)^{-1}CA^{r+1}\nu+\lambda^{r+1}(CA^rH)^{-1}C\nu.\label{sqcom2}
\end{align}
Comparing (\ref{sqcom2}) with (\ref{sqcom}), it follows that
\begin{align}
 \lambda^{r+1}(CA^rH)^{-1}C\nu=0.\nonumber
\end{align}
Since $\lambda\neq0$, and $(CA^rH)^{-1}$ is full rank, it follows that
\begin{align}
 C\nu=0.\label{sqim2}
\end{align}
Combining (\ref{sqim1}) and (\ref{sqim2}), we have
\begin{align}
\left[ \begin{array}{cc}
 (\lambda I-A) & -H  \\
 C & 0
\end{array}  \right]
\left[ \begin{array}{cc}
 \nu  \\
   \mu  \\
 \end{array}  \right] =0. \label{eigproofA.18}
\end{align}
Since $\nu \neq 0$, it follows that $\left[ \begin{array}{cc}
 \nu  \\
   \mu \\
 \end{array}  \right] \neq 0 $ and noting that normal rank $Z(s) = n+p$ from
 Corollary 12.10.6 in \cite{bernsteinbook}, it follows that
\begin{align}
{\rm rank}\left[ \begin{array}{cc}
 (\lambda I-A) & -H  \\
 C & 0
\end{array}  \right]<{\rm normal \: rank}\left[ \begin{array}{cc}
 (z I-A) & -H  \\
 C & 0
\end{array}  \right]= n+p . \label{eigproofA.19}
\end{align}
Therefore $ \lambda$ is an invariant zero of $(A,H,C)$.\qed\\
In Lemma (\ref{lem:eig_r}), we established the relationship between
the invariant zeros and the eigenvalues of $(A-L_kCA^{r+1})$. Next, we
use this relationship to examine the convergence of the filter.
\begin{thm} \label{thm:suff_r} Let $l=p$ and let $L_k$ be such that
  (\ref{cond_r}) hold. Then the filter (\ref{xhat1r}) - (\ref{xhat3r})
  is unbiased if and only if (\ref{x}), (\ref{y}) have no invariant
  zeros, while the filter (\ref{xhat1r}) - (\ref{xhat3r}) is
  asymptotically unbiased if and only if all of the invariant zeros
  lie within the unit circle.
\end{thm}
{\bf Proof.}
First, taking the expected values of both sides of (\ref{eq:error_r}) and using
(\ref{cond_r}), and noting that $
\mathbb{E}[w_{k-r-1}] = \mathbb{E}[v_k] = \mathbb{E}[w_{k-r}]=
\mathbb{E}[w_{k-1}] = 0,$ it follows
that
\begin{align}
\mathbb{E}[\varepsilon_{k-r}] =
\mathbb{E}[(A-LCA^{r+1})\varepsilon_{k-r-1}]=(A-LCA^{r+1})\mathbb{E}[
\varepsilon_ {
k-r-1} ] .
\label{ers}
\end{align}
Therefore, filter (\ref{xhat1r}) - (\ref{xhat3r}) is unbiased if and
only if $(A-LCA^{r+1})$ is zero or equivalently all eigenvalues of
$(A-LCA^{r+1})$ are zero. Furthermore, (\ref{xhat1r}) - (\ref{xhat3r})
is asymptotically unbiased if and only if $(A-LCA^{r+1})$ is
asymptotically stable.  It follows from Lemma \ref{lem:eig_r} that all
non-zero eigenvalues of $(A-LCA^{r+1})$ are invariant zeros of
(\ref{x}), (\ref{y}). Subsequently, the filter (\ref{xhat1r}) -
(\ref{xhat3r}) is asymptotically unbiased if and only if no invariant
zeros of (\ref{x}) and (\ref{y}) are non-minimum-phase and there are
no zeros on the unit circle (all of the invariant zeros are within the
unit circle). The filter gives rise to a persistent reconstruction error if the
zeros lie on the unit circle. \qed
\subsection{Input Reconstruction}

We discussed necessary and sufficient conditions to obtain the unbiased
estimates of states. Next, we consider using these estimates to reconstruct the
unknown inputs.

\begin{prop}\label{prop_r}
Let $l=p$ and let $L_k$ be such that (\ref{cond_r}) hold, and let $\hat
x_{k-r|k}$ be an unbiased estimate of $x_{k-r}$. Then
\begin{align}
\hat e_{k-r-1} \isdef (CA^rH)^{-1}(y_{k} - C\hat x_{k|k-1}) \label{IR_r}
\end{align}
is an unbiased estimate of $e_{k-r-1}$.
\end{prop}
{\bf Proof.} Since rank$(CA^rH) = p$, we can define
\begin{align}
\hat e_{k-r-1} \isdef (CA^rH)^{-1}(y_{k} - C\hat x_{k|k-1}). \label{eq:eest1}
\end{align}
Next, using (\ref{eq:eest1}), (\ref{xhat1r}), (\ref{see_r}),
(\ref{eq:L_squarer}) it follows that
\begin{align}
\hat e_{k-r-1} & = (CA^rH)^{-1}CA^rH(CA^rH)^{-1} (y_{k} - C\hat x_{k|k-1})
\nonumber \\
& = (CA^rH)^{-1} CA^rL_{k}(y_{k} - C\hat x_{k|k-1}) \nonumber \\
& = (CA^rH)^{-1} CA^r(\hat x_{k-r|k} - \hat x_{k-r|k-1}) \nonumber \\
& = (CA^rH)^{-1} CA^r(x_{k-r} + \varepsilon_{k-r} -\hat x_{k-r|k-1})
\nonumber \\
& = (CA^rH)^{-1} CA^r(Ax_{k-r-1} + He_{k-r-1} + w_{k-r-1} + \varepsilon_{k-r} -
A \hat x_{k-r-1|k-1}) \nonumber \\
& = (CA^rH)^{-1} CA^r(A\varepsilon_{k-r-1} + He_{k-r-1} + w_{k-r-1} +
\varepsilon_{k-r}).
\end{align}
Finally, taking the expected values of both sides and noting that 
\begin{align}
\mathbb{E}[\varepsilon_{k-r}] &= \mathbb{E}[\varepsilon_{k-r-1}] =
\mathbb{E}[w_{k-r-1}] = 0,
\end{align}
 it follows that
\begin{align*}
\mathbb{E}[\hat e_{k-r-1}] = (CA^rH)^{-1} CA^rH \mathbb{E}[e_{k-r-1}] =
\mathbb{E}[e_{k-r-1}]. \tag*{$\square$}
\end{align*}
\subsection{Non-square Systems}\label{non-sqr}
In the previous subsection we dealt with square systems.
In this section we explore the possibility of input
reconstruction with a delay for non-square systems $(l\neq p)$. 
In the context of non square systems, first we show that the
invariant-zeros of the non-square system are a subset of the eigenvalues of
$(A-L_{k}CA^{r+1})$ as follows.

\begin{lemma} \label{lem:eign}
Let $L_{k}$ be such that (\ref{cond_r}) is
satisfied. Then the invariant zeros of system (\ref{x}), (\ref{y}) are a subset
of the eigenvalues of $(A-L_{k}CA^{r+1})$.
\end{lemma}
{\bf Proof.} 
Let $z$ be an invariant zero of (\ref{x}), (\ref{y}), and let vector
$\left[ \begin{array}{cc} \nu  \\
   \mu \\
 \end{array}  \right] \neq 0 $ be such that
\begin{align} 
\left[ \begin{array}{cc}
 (zI-A) & -H  \\
 C & 0
\end{array}  \right]
\left[ \begin{array}{cc}
 \nu  \\
   \mu  \\
 \end{array}  \right] =0.
\label{n1}
\end{align}
Therefore
\begin{align}
(zI-A)\nu - H\mu&=0,\label{n2}\\
C\nu&=0.\label{n3}
\end{align}
Note that if $\nu=0$, it can be seen from (\ref{n2}) that $H\mu=0$, but since
rank$(H) = p$, $\mu=0$ violating the assumption 
$\left[ \begin{array}{cc} \nu 
\\
   \mu \\
 \end{array}  \right] \neq 0 .$ Hence $\nu\neq 0$.\\
 
Next, left multiplying (\ref{n2}) by $L_{k}CA^r$ and rearranging,
\begin{align}
zL_{k}CA^r\nu-L_{k}CA^{r+1}\nu=L_{k}CA^rH\mu.\label{n6_r}
 \end{align}
Also it follows from (\ref{n2}) that,
\begin{align}
 A\nu=z\nu-H\mu. \label{rear}
\end{align}
Next using (\ref{rear}) in (\ref{n6_r}) and rearranging, 
\begin{eqnarray}
 \lefteqn{z^{r+1}L_{k}C\nu-z^rL_{k}CH\mu-z^{r-1}L_{k}
CAH\mu-\hdots} {}\nonumber \\&&{}-zL_{k}
CA^{r-1}H\mu-L_{k}CA^{r+1}\nu=L_{k}CA^rH\mu.
\end{eqnarray}
and using (\ref{n3}), (\ref{cond_r}),
\begin{align}
-L_{k}CA^{r+1}\nu=H\mu.\label{n7}
\end{align} 
Using (\ref{n7}) in (\ref{n2}),
\begin{align}
(zI-A)\nu + L_{k}CA^{r+1}\nu = 0,\nonumber
\end{align} 
\begin{align}
zI\nu-A\nu + L_{k}CA^{r+1}\nu = 0,\nonumber
\end{align}
and further rearranging, we have
\begin{align}
(A-L_{k}CA^{r+1}-zI)\nu=0.
\end{align}
Since $\nu\neq$ 0, It follows that $z$ is an eigenvalue of
$(A-L_{k}CA^{r+1}).$
\qed\\
For a non-square system, we note however that the converse of Lemma
\ref{lem:eign} does not hold as the following counter example demonstrates.\newline
Consider a state space system characterized by the following $A,H,C$
matrices.
\begin{gather*}\label{counterex}
\begin{align}
A &=\matl{cccccccccccc}-0.95 & -0.04 & 0 & 0 & 0 & 0 & 0 & 0 & 0 & 0
  & 0 & 0\\
 0.025 & 1 & 0 & 0 & 0 & 0 & 0 & 0 & 0 & 0
  & 0 & 0\\
  0 & 0 & 0.97 & -0.06 & 0 & 0 & 0 & 0 & 0 & 0
  & 0 & 0\\
  0 & 0 & 0.05 & 1 & 0 & 0 & 0 & 0 & 0 & 0
  & 0 & 0\\
 0 & 0 & 0 & 0 & 0.95 & -0.05 & 0 & 0 & 0 & 0
  & 0 & 0\\
0 & 0 & 0 & 0 & 0.1 & 1 & 0 & 0 & 0 & 0
  & 0 & 0\\
0 & 0 & 0 & 0 & 0 & 0 & 0.98 & -0.04 & 0 & 0
  & 0 & 0\\
0 & 0 & 0 & 0 & 0 & 0 & 0.05 & 1 & 0 & 0
  & 0 & 0\\
0 & 0 & 0 & 0 & 0 & 0 & 0 & 0 & 0.95 & -0.08
  & 0 & 0\\
0 & 0 & 0 & 0 & 0 & 0 & 0 & 0 & 0.05 & 1
  & 0 & 0\\
0 & 0 & 0 & 0 & 0 & 0 & 0 & 0 & 0 & 0
  & 0.95 & -0.06\\
0 & 0 & 0 & 0 & 0 & 0 & 0 & 0 & 0 & 0
  & 0.1 & 1
    \matr, 
     \quad H = \matl{cc} 0.4 & 0\\0 & 0\\0.2 & 0\\0 & 0\\0.2 & 0\\0 &
 0\\0 & 0.2 \\0 & 0\\0 & 0.2\\0 & 0\\0 & 0.2\\0 & 0\matr,\notag \\
   C&=\matl{cccccccccccc} 0.25 & 2 & 0 & 0 & 0 & 0 & 0.5 & 2 & 0 & 0 & 0 &
    0\\0 & 0 & 0.5 & 2 & 0 & 0 & 0 & 0 & 0.5 & 2 & 0 &
    0\\0 & 0 & 0 & 0 & 0.5 & 1 & 0 & 0 & 0 & 0 & 0.5 &
    1 \matr.
\end{align}
\end{gather*}  
It is seen that this system has an invariant zero at $0.8$ of
multiplicity two. The eigenvalues of $(A-L_kCA^2)$, a delay of
one-time step, ($r=1$), are found to be $0.7528,\,
0.9782+0.0707i,\,0.9782-0.0707i,\,0.9750 + 0.0443i,\,0.9750 -
0.0443i,\,1.0076,\,0.8,\,0.8.$ It can be seen that the invariant zeros
of the non-square system are the eigenvalues of $(A-L_kCA^2)$ in
accordance with lemma \ref{lem:eign}, but along with other spurious
eigenvalue of which are not the invariant zeros of the system. This
result is obtained using a value of $L_k$ that satisfies
(\ref{cond_r}). The calculation of $L_k$ in the example \eqref{counterex} is
based on the procedure which is discussed next.  
\par We note that in the non-square case, an infinite number of solutions
for $L_{k}$ that satisfy (\ref{cond_r}) are possible. The following
results thus derive the $L_{k}$ that minimizes the trace of the error
covariance and hence the minimum variance gain. $Q_k$ and $R_k$ are
the process noise covariance and sensor noise covariance respectively.
\begin{fact} \label{fact:Pgd} Let $L_{k}$ be such that the filter 
(\ref{xhat1r}) - (\ref{xhat3r}) is unbiased. Then
\begin{eqnarray}
\lefteqn{ P_{k-r\mid k} =  (A-L_{k}CA^{r+1})P_{k-r-1|k-1}
(A-L_{k}CA^{r+1})^\rmT }
{}\nonumber\\
&& {}+ (I-L_{k}CA^r)Q_{k-r-1}(I-L_{k}CA^r)^\rmT +
(L_{k}CA^{r-1})Q_{k-r}(L_{k}CA^{r-1})^\rmT+{}\nonumber\\
 && {} \hdots +(L_{k}C)Q_{k-1}(L_{k}C)^\rmT+
L_{k}R_{k}L_{k}^\rmT \label{eq:Pgd}
\end{eqnarray}
\end{fact}
{\bf Proof.} The proof follows by substituting (\ref{eq:error_r}) in (\ref{ecmg})
and using (\ref{cond_r}). \qed\\

Next, define the cost function $J$ as the trace of the error covariance matrix
\begin{align}
      J(L_{k}) =
\mathrm{tr}\mathbb{E}[\varepsilon_{k-r}\varepsilon^{\mathrm{T}}_{k-r}] =
\mathrm{tr}P_{k-r\mid k}. \label{cost}
\end{align}
Therefore, it follows from (\ref{eq:Pgd}) that
\begin{eqnarray}
\lefteqn{J(L_{k}) = \mathrm{tr}\big[ (A-L_{k}CA^{r+1})P_{k-r-1|k-1}
(A-L_{k}CA^{r+1})^\rmT } 
{}\nonumber\\
&& {}+ (I-L_{k}CA^r)Q_{k-r-1}(I-L_{k}CA^r)^\rmT +
(L_{k}CA^{r-1})Q_{k-r}(L_{k}CA^{r-1})^\rmT+ {}\nonumber\\
 && {} \hdots +(L_{k}C)Q_{k-1}(L_{k}C)^\rmT+
L_{k}R_{k}L_{k}^\rmT \big].\label{cost_r}
\end{eqnarray}

To derive the unbiased minimum-variance filter gain, we minimize the
objective function (\ref{cost_r}) subject to the constraints
(\ref{cond_r}) while noting that from Corollary \ref{cor_r}, we have 
rank$(CA^dH) \leq l-p$.

\begin{thm}\label{Lthm} Suppose there exists at least one $L_k$ that satisfies
(\ref{cond_r}), then the unbiased minimum-variance gain $L_{k}$ is
\begin{eqnarray}
L_{k}= \left[ T_{k}A^{r^\rmT} C^\rmT +
N_kZ_k^{\dagger}\big[CA^rH\,\,\,\,CA^{r-1}H\,\,\,\,
\hdots\,\,\,\,CH\big]^\rmT \right]
S_{k}^{-1}. \label{Minvar_gain}
\end{eqnarray}
where
\begin{align*}
T_{k} & \isdef Q_{k-r-1}+AP_{k-r-1|k-1}A^{\rmT}, \\
S_{k} & \isdef CA^rT_{k}A^{r^\rmT} C^\rmT
+CA^{r-1}Q_{k-r}A^{{r-1}^\rmT}C^\rmT+\hdots+ CQ_{k-1}C^\rmT + R_{k},
\end{align*}
\begin{align*}
 Z_k=
\matl{cccc}(CA^rH)^\rmT S_{k}^{-1}(CA^rH)&(CA^rH)^\rmT S_{k}^{-1}(CA^{r-1}
H)&\hdots&(CA^rH)^\rmT S_{k}^{-1}(CH)\\(CA^{r-1}H)^\rmT S_{k}^{-1}
(CA^rH)&(CA^{r-1}H)^\rmT S_{k}^{-1}(CA^{r-1}
H)&\hdots&(CA^{r-1}H)^\rmT S_{k}^{-1}(CH)\\\vdots&\vdots&\ddots&\vdots\\
(CH)^\rmT S_{k}^{-1}(CA^rH)&(CH)^\rmT S_{k}^{-1}(CA^{r-1}H)
&\hdots&(CH)^\rmT S_{k}^{-1}(CH)\matr,
\end{align*}
and
\begin{align*}
N_k&=\matl{cccc}H-T_{k}A^{r^\rmT}C^\rmT S_{k}^{-1}CA^rH
&-T_{k}A^{r^\rmT}C^\rmT S_{k}^{-1}CA^{r-1}H &\hdots&-T_{k}A^{r^\rmT}C^\rmT
S_{k}^{-1}CH\matr.
\end{align*}
\end{thm}
{\bf Proof.} The Lagrangian for the constrained minimization problem is
\begin{eqnarray}
\lefteqn{\mathcal{L}(L_{k}) \triangleq J(L_{k})}{}\nonumber\\
&&{}+2\mathrm{tr}\big(\big[(I-L_{k}CA^r)H\,\,\,\,
L_{k}CA^{r-1}H\,\,\,\,
\hdots\,\,\,\,L_{k}CH\big]\Lambda_k\big),\label{eq:rd1_Lag_r}
\end{eqnarray}
where $\Lambda_{k}^\rmT \in \BBR^{n \times (r+1)p}$ is the
matrix of Lagrange multipliers.
Next, differentiating (\ref{eq:rd1_Lag_r}) with respect to $L_{k}$ and setting
it equal to zero yields
\begin{eqnarray}
\lefteqn{ -2CA^rQ_{k-r-1} -2CA^{r+1}P_{k-r-1|k-1}A^\rmT +2\big[
CA^{r+1}P_{k-r-1|k-1}A^{{r+1}^\rmT}C^\rmT} {}\nonumber\\
&& +CA^rQ_{k-r-1}A^{r^\rmT} C^\rmT+CA^{r-1}Q_{k-r}A^{{r-1}^\rmT}C^\rmT+\hdots
+CQ_{k-1}C^\rmT +R_{k}\big] L_{k}^\rmT {}\nonumber\\
&&-2\big[CA^rH\,\,\,\,CA^{r-1}H\,\,\,\,
\hdots\,\,\,\,CH\big]\Lambda_k=0. \label{eq:rd_2r}
\end{eqnarray}
Next assuming $S_k$ is invertible and solving (\ref{eq:rd_2r}) for $ L_{k}$, we
get
\begin{align}
L_{k}= \left[ T_{k}A^{r^\rmT} C^\rmT +
\Lambda_{k}^\rmT\big[CA^rH\,\,\,\,CA^{r-1}H\,\,\,\,
\hdots\,\,\,\,CH\big]^\rmT \right]
S_{k}^{-1}.
\label{rd_1Ltrans_r}
\end{align}
Next, to solve for $\Lambda_{k}$, we substitute (\ref{rd_1Ltrans_r}) in
(\ref{cond_r}) to get
\begin{align}
\left[ T_{k}A^{r^\rmT} C^\rmT + \Lambda_{k}^\rmT\big[CA^rH\,\,\,\,
CA^{r-1}H\,\,\,\,\hdots\,\,\,\,CH\big]^\rmT \right]S_{k}^{-1}
\big[CA^rH\,\,\,\,CA^{r-1}H\,\,\,\,\hdots\,\,\,\,CH\big]={}\nonumber\\
\big[H\,\,\,\,0_{n\times p}\,\,\,\,\hdots\,\,\,\,0_{n\times p}\big]. \label{L_p}
\end{align}

Next we define $Z_k$ and $N_k$ as follows
\begin{align}
Z_k=
\matl{cccc}(CA^rH)^\rmT S_{k}^{-1}(CA^rH)&(CA^rH)^\rmT S_{k}^{-1}(CA^{r-1}
H)&\hdots&(CA^rH)^\rmT S_{k}^{-1}(CH)\\(CA^{r-1}H)^\rmT S_{k}^{-1}
(CA^rH)&(CA^{r-1}H)^\rmT S_{k}^{-1}(CA^{r-1}
H)&\hdots&(CA^{r-1}H)^\rmT S_{k}^{-1}(CH)\\\vdots&\vdots&\ddots&\vdots\\
(CH)^\rmT S_{k}^{-1}(CA^rH)&(CH)^\rmT S_{k}^{-1}(CA^{r-1}H)
&\hdots&(CH)^\rmT S_{k}^{-1}(CH)\matr,
\end{align}

\begin{align}
 N_k&=\matl{cccc}H-T_{k}A^{r^\rmT}C^\rmT S_{k}^{-1}CA^rH
&-T_{k}A^{r^\rmT}C^\rmT S_{k}^{-1}CA^{r-1}H &\hdots&-T_{k}A^{r^\rmT}C^\rmT
S_{k}^{-1}CH\matr.
\end{align}
Therefore (\ref{L_p}) becomes
\begin{align}
 \Lambda_k^\rmT Z_k&=N_k.
\end{align}
Solving for $\Lambda_k^\rmT$ we get
\begin{align}
 \Lambda_k^\rmT=N_kZ_k^{\dagger}.\label{value:lambda}
\end{align}
Where $Z_k^{\dagger}$ is a Moore-Penrose generalized inverse of
$Z_k$. 
Substituting (\ref{value:lambda}) in (\ref{rd_1Ltrans_r}) we get 
\begin{align}
     L_{k}= \left[ T_{k}A^{r^\rmT} C^\rmT +
N_kZ_k^{\dagger}\big[CA^rH\,\,\,\,CA^{r-1}H\,\,\,\,
\hdots\,\,\,\,CH\big]^\rmT \right]
S_{k}^{-1}.
\label{rd_2Ltrans_r} \tag*{$\square$}
\end{align}
Note that the assumption of $S_k$ being invertible is ensured by
demanding apriori that $R_k$ is positive definite for all $k$. This
indicates the persistence of sensor noise and is a valid assumption on
physical grounds. $Q_k$ is assumed to be nonnegative definite for all
$k$.  It should also be noted that since $Z_k$ is not full rank there
are an infinitely many possible solutions for $\Lambda_k$. However at
any instant $k$, any $\Lambda_k$ that satisfies (\ref{value:lambda}) will give the
same minimum-variance gain ($L_k$).
\begin{cor}\label{lcor} Suppose $CA^dH = 0$ for $d=0,\,1,\,\hdots,\,r$ 
then the unbiased minimum-variance gain
$L_{k}$ is
\begin{eqnarray}
L_{k}= \left[ T_{k}A^{r^\rmT} C^\rmT + \Phi_{k} (CA^rH)^\rmT \right]
S_{k}^{-1},
\end{eqnarray}
where
\begin{align}
T_{k} & \isdef Q_{k-r-1}+AP_{k-r-1|k-1}A^{\rmT}, \\
S_{k} & \isdef CA^rT_{k}A^{r^\rmT} C^\rmT
+CA^{r-1}Q_{k-r}A^{{r-1}^\rmT}C^\rmT+\hdots+ CQ_{k-1}C^\rmT + R_{k}, \\
\Phi_{k} & \isdef \left[ H - T_{k}A^{r^\rmT} C^\rmT S_{k}^{-1} CA^rH
\right] \left( (CA^rH)^\rmT
S_{k}^{-1} CA^rH\right)^{-1}
\end{align}
\end{cor}

In the next section we present some numerical results using the previously
developed filter.
\section{Existence and maximum delay}
 According to (\ref{cond_r}), $L_k$ should satisfy the condition
 \begin{align}
 L_k\begin{bmatrix}CA^{r}H & CA^{r-1}H & \hdots &
   CH\end{bmatrix}=\begin{bmatrix}H & 0 &\hdots &0 
 \end{bmatrix}\label{existence_L}.
 \end{align}

\begin{lem}
$L_k\begin{bmatrix}CA^rH & CA^{r-1}H \hdots
  CH\end{bmatrix}\neq\begin{bmatrix}H & 0 & \hdots & 0\end{bmatrix} \;\;
\forall\;\;r \geq n$
\end{lem}
{\bf Proof.}
Suppose on the contrary, assume $L_k\begin{bmatrix}CA^rH & CA^{r-1}H \hdots
  CH\end{bmatrix}=\begin{bmatrix}H & 0 & \hdots & 0\end{bmatrix} \;\;
\forall\;\;r \geq n.$ Let $r=n.$ The condition for unbiasedness
becomes  
\begin{align}
L_k\begin{bmatrix}CA^nH & CA^{n-1}H \hdots
  CH\end{bmatrix}=\begin{bmatrix}H & 0 & \hdots & 0\end{bmatrix}\label{lemma:eq1}
\end{align}
Using Cayley-Hamilton theorem and after rearranging \eqref{lemma:eq1} becomes
\begin{align}
\begin{bmatrix}(d_1L_kCA^{n-1}H+d_2L_kCA^{n-2}H+\hdots+d_nL_kCH) & L_kCA^{n-1}H \hdots
  L_kCH\end{bmatrix}=\begin{bmatrix}H & 0 & \hdots & 0\end{bmatrix}, 
\end{align}
which results in a contradiction, which is also true for all $r > n$. Hence our assumption is wrong and
therefore 
\begin{align}
L_k\begin{bmatrix}CA^rH & CA^{r-1}H \hdots
  CH\end{bmatrix}\neq\begin{bmatrix}H & 0 & \hdots & 0\end{bmatrix} \;\;
\forall\;\;r > n \tag*{$\square$}
\end{align}
This shows that the choice of the delay, $r$, cannot be arbitrary and
has to be smaller than the system order. The lower bound on the delay
is determined by the first Markov parameter which has full rank. \newline
\begin{con}
The filter will be unbiased for only one value of the delay $r$,
i.e. there exists only one value of $r$ that satisfies condition
\eqref{existence_L}. 
\end{con}
Next, we introduce the following notations 
  \begin{align*}
  S_r&\isdef\begin{bmatrix}CA^{r}H & CA^{r-1}H & \hdots &
    CH\end{bmatrix}, \\
  S_{r-1}&\isdef\begin{bmatrix}CA^{r-1}H & CA^{r-2}H & \hdots &
    CH\end{bmatrix}.
 \end{align*}
\begin{thm}\label{thm:existence}
There exists a matrix $L_k$ that satisfies (\ref{existence_L}) if and
only if 
\begin{align}
\mathrm{rank}\begin{bmatrix}S_r\end{bmatrix}-\mathrm{rank}\begin{bmatrix}S_{r-1}
\end{bmatrix}=p
\end{align}
\end{thm}
{\bf Proof.}
There exists a matrix $L_k$ that satisfies (\ref{existence_L}) if
 and only if the matrix $\begin{bmatrix}H & 0 & \dots &
  0\end{bmatrix}$ is in the space spanned by the rows of $S_r$. This
is equivalent to the condition 
\begin{align}
\mathrm{rank}
\begin{bmatrix*}S_r\\
H\;\;\; 0\;\;\; \hdots\;\;\; 0 \end{bmatrix*}
=\mathrm{rank} \begin{bmatrix}S_r\end{bmatrix}\label{ex1}
\end{align}.
\begin{align*}
\mathrm{rank}
\left[\begin{array}{c|ccc}
CA^rH & CA^{r-1}H & \hdots & CH\\
H & 0 & \hdots & 0
\end{array}\right]
\end{align*}

Using the matrix identity 
\begin{align}
\mathrm{rank}\begin{bmatrix}A &
B\end{bmatrix}=\mathrm{rank}\begin{bmatrix}A\end{bmatrix}+\mathrm{rank}\begin{bmatrix}B\end{bmatrix}-\mathrm{dim}\Big(\mathcal{R}\begin{bmatrix}A\end{bmatrix}\cap\mathcal{R}\begin{bmatrix}B\end{bmatrix}\Big),\label{matrix.identity:con}
\end{align}
and noting that $\mathrm{rank}[H]=p$ we get
\begin{align}
\mathrm{rank}
\begin{bmatrix*}S_r\\
H\;\;\; 0\;\;\; \hdots\;\;\;
0 \end{bmatrix*}=\mathrm{rank}\begin{bmatrix}S_{r-1}\end{bmatrix} + p \label{ex2}.
\end{align}
From (\ref{ex1}) and (\ref{ex2}) we get the condition
\begin{align}
\mathrm{rank}\begin{bmatrix}S_r\end{bmatrix}=\mathrm{rank}\begin{bmatrix}S_{r-1}\end{bmatrix}
  + p. \label{ex3} \tag*{$\square$}
\end{align}
\begin{cor}\label{existence:cor}
If
$\mathrm{rank}\begin{bmatrix}S_r\end{bmatrix}-\mathrm{rank}\begin{bmatrix}S_{r-1}\end{bmatrix}=p$
then
$\mathrm{rank}\begin{bmatrix}M_r\end{bmatrix}-\mathrm{rank}\begin{bmatrix}M_{r-1}\end{bmatrix}=p,$
where 
\begin{align*}M_r=
\begin{bmatrix}
CH & 0 & 0 & \hdots  & 0\\
CAH & CH & 0 & \hdots & 0\\
\vdots & \vdots & \vdots & \ddots & \vdots\\
CA^{r}H & CA^{r-1}H & CA^{r-2}H & \hdots & 0
\end{bmatrix}.
\end{align*}
\end{cor}
\begin{proof}
Using (\ref{matrix.identity:con}) we have
\begin{align}
\mathrm{rank}\begin{bmatrix}M_r\end{bmatrix}-\mathrm{rank}\begin{bmatrix}M_{r-1}\end{bmatrix}=\mathrm{rank}\begin{bmatrix}
CH\\CAH\\\vdots\\CA^{r}H\end{bmatrix}+\mathrm{dim}\Big(\mathcal{R}\begin{bmatrix}CH\\CAH\\\vdots\\CA^{r}H\end{bmatrix}\cap\mathcal{R}\begin{bmatrix}M_{r-1}\end{bmatrix}\Big),\label{M:identity}
\end{align}
where $M_{r-1}$ can also be represented as follows
\begin{align}
\begin{bmatrix}M_{r-1}\end{bmatrix}=
\left[\begin{array}{cccc}
0 & 0 & \hdots & 0\\
CH & 0 & \hdots & 0\\
CAH & CH & \hdots & 0\\
\vdots & \vdots & \vdots & \vdots\\
\hline\\
\multicolumn{4}{c}{\raisebox{.5\normalbaselineskip}{$S_{r-1}$}}
\end{array}\right].\label{inter:0}
\end{align}
Since, 
$\mathrm{rank}\begin{bmatrix}S_r\end{bmatrix}-\mathrm{rank}\begin{bmatrix}S_{r-1}\end{bmatrix}=p$
according the Theorem \ref{thm:existence}, there exists a matrix $L_k$
such that (\ref{existence_L}) and hence (\ref{cond_r}) is
satisfied. Then from Corollary \ref{cor_r}, $ii)$ we have
$\mathrm{rank}(CA^{r}H)=p.$ Therefore
\begin{align}
\mathrm{rank}\begin{bmatrix}
CH\\CAH\\\vdots\\CA^{r}H\end{bmatrix}=p.\label{inter:1} 
\end{align} 
Next, Using (\ref{matrix.identity:con}) we have
\begin{align}
\mathrm{rank}\begin{bmatrix}S_r\end{bmatrix}-\mathrm{rank}\begin{bmatrix}S_{r-1}\end{bmatrix}=\mathrm{rank}\begin{bmatrix}CA^rH\end{bmatrix}+\mathrm{dim}\Big(\mathcal{R}\begin{bmatrix}CA^rH\end{bmatrix}\cap\mathcal{R}\begin{bmatrix}S_{r-1}\end{bmatrix}\Big).\label{inter:2}
\end{align}
Now since,
$\mathrm{rank}\begin{bmatrix}S_r\end{bmatrix}-\mathrm{rank}\begin{bmatrix}S_{r-1}\end{bmatrix}=p$
and $\mathrm{rank}(CA^rH)=p$, (\ref{inter:2}) leads to 
\begin{align}
\mathrm{dim}\Big(\mathcal{R}\begin{bmatrix}CA^rH\end{bmatrix}\cap\mathcal{R}\begin{bmatrix}S_{r-1}\end{bmatrix}\Big)=0.\label{inter:3}
\end{align}
From (\ref{M:identity}), (\ref{inter:0}) and (\ref{inter:3}) we have
\begin{align}
\mathrm{dim}\Big(\mathcal{R}\begin{bmatrix}CH\\CAH\\\vdots\\CA^{r}H\end{bmatrix}\cap\mathcal{R}\begin{bmatrix}M_{r-1}\end{bmatrix}\Big)=0. \label{inter:4}
\end{align}
Next, substituting (\ref{inter:1}) and (\ref{inter:4}) in
(\ref{M:identity}) leads to 
\begin{align*}
\mathrm{rank}\begin{bmatrix}M_r\end{bmatrix}-\mathrm{rank}\begin{bmatrix}M_{r-1}\end{bmatrix}=p.
\end{align*}

\end{proof}


This in turn implies that the system (\ref{x}) and (\ref{y}) is
$r$-delay invertible \cite{willsky74}. However it should be noted that the $r$-delay invertibility of a
system is not a sufficient condition for the existence of $L_k$ which
satisfies the condition (\ref{existence_L}) as the following example
shows.
Consider the system
\begin{align*}A=
\begin{bmatrix}
0.5 & -0.6 & 0 & 0\\0.5 & 0 & 0 & 0\\0 & 0 & -0.5 & -0.6\\0 & 0 & 0.5
& 0
\end{bmatrix},\,H=\begin{bmatrix}4 & 0\\0 & 0\\ 0 & 4\\0 &
  0\end{bmatrix},\,C=\begin{bmatrix}
0.25 & 1.05 & 0.25 & 1.1 \\ 0.25 & 1.15 & 0.25 & 1\\0.25 & 1.05 & 0.25 &1.1
\end{bmatrix}.
\end{align*}
For this system $\mathrm{rank}(M_r)-\mathrm{rank}(M_{r-1})=2=p$, 
however $\mathrm{rank}(S_r)-\mathrm{rank}(S_{r-1}) < p.$\newline 


\section{ Numerical Results Using an Unbiased Filter with Delay of One Time
Step}

\subsection{Numerical Results - Square Systems} \label{comp-example}
\begin{figure}[ht!]
\centering{
\includegraphics[width=5cm]{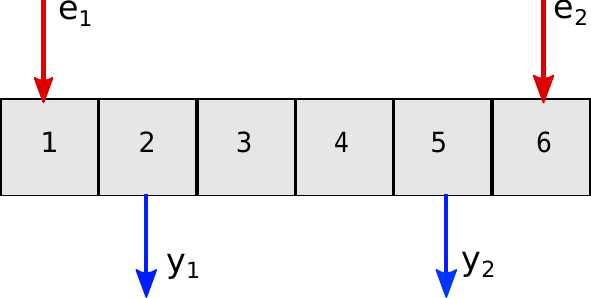}}
\caption{This compartmental model comprises of six compartments that exchange
mass or energy through mutual interaction. The inputs are provided to the first
and the last compartment while the states from the second and the fifth
compartment are measured and constitute the outputs.}
\label{fig:cm}
\end{figure}
To illustrate recursive input reconstruction, we consider a
compartmental system comprised of $n$ compartments that exchange mass
or energy through mutual interaction. This can represent physical
models like collection of rooms which mutually exchange mass and energy.
The conservation equations governing the compartmental model are
\begin{align}
x_{1,k+1} & =  x_{1,k} - \beta x_{1,k} + \alpha
(x_{2,k} -
x_{1,k}), \label{eq:basic_compartmental1}\\
x_{i,k+1} & =  x_{i,k} - \beta x_{i,k} + \alpha
(x_{i+1,k} - x_{i,k}) -\alpha (x_{i,k} - x_{i-1,k}),\nonumber \\& \quad i =
2,\,\dots,\, n-1,\label{eq:basic_compartmental}\\
x_{n,k+1} & = x_{n,k} - \beta x_{n,k} -\alpha
(x_{n,k} - x_{n-1,k}), \label{eq:basic_compartmental3}
\end{align}
where $0 < \beta < 1$ is the loss coefficient and $0 < \alpha < 1$
is the flow coefficient.
Figure \ref{fig:cm} illustrates a schematic of the compartmental model wherein
each block represents a compartment. The arrows labeled $e_1$, $e_2$ indicate the input to the
system while the arrows labeled $y_1$, $y_2$ indicate the output. 
The system equations
(\ref{eq:basic_compartmental1}) - (\ref{eq:basic_compartmental3}) can be written
in the state space form (\ref{x}), (\ref{y}) with
\begin{align}
A & = \matl{ccccc} 1-\beta-\alpha & \alpha & 0 & \cdots & 0 \\
    \alpha & 1-\beta-\alpha & \alpha & \cdots & 0 \\
    \vdots & & \ddots & \ddots & \vdots \\
    0 & \cdots & 0 & \alpha & 1-\beta-\alpha \matr. \label{eq:ex_A}
\end{align}

Further, for the numerical simulations, we choose $n=6$, $\alpha = 0.1$ and
$\beta = 0.1$, we assume we have no known inputs and therefore the $Bu_k$ and
$Du_k$ terms disappear, and we assume two unknown inputs enter compartments 1
and 6, while the states in compartments 2 and 5 are measured as outputs. It then
follows that
\begin{align}
H = \matl{cccccc} 1 & 0 & 0 & 0 & 0 & 0 \\ 0 & 0 & 0 & 0 & 0 & 1
\matr^\rmT, \quad C & =
\matl{cccccc} 0 & 1 & 0 & 0 & 0 & 0 \\
    0 & 0 & 0 & 0 & 1 &  0 \matr. \label{eq:ex_comp2}
\end{align}
Note that since $CH=0,$ the filters in
\cite{kitanidis,steven_kitanidis,darouach94,palanthACC2006} cannot be applied
for input reconstruction (or state estimation). Further, since $CAH$ is full
rank and there are no invariant zeros of the system, it follows from Theorem
\ref{thm:suff_r} that the filter  (\ref{xhat1r}) - (\ref{xhat3r}),
(\ref{eq:L_squarer}), (\ref{IR_r}) with $r=1$ will provide an unbiased estimate
of the unknown inputs.
We choose the first input to be a sawtooth and the second input
to be a sinusoid. Figure \ref{fig:min1a} shows the actual
unknown inputs and the estimated unknown inputs using the recursive filter
developed previously.
\begin{figure}[ht!]
\centering{
\includegraphics[scale=0.4]{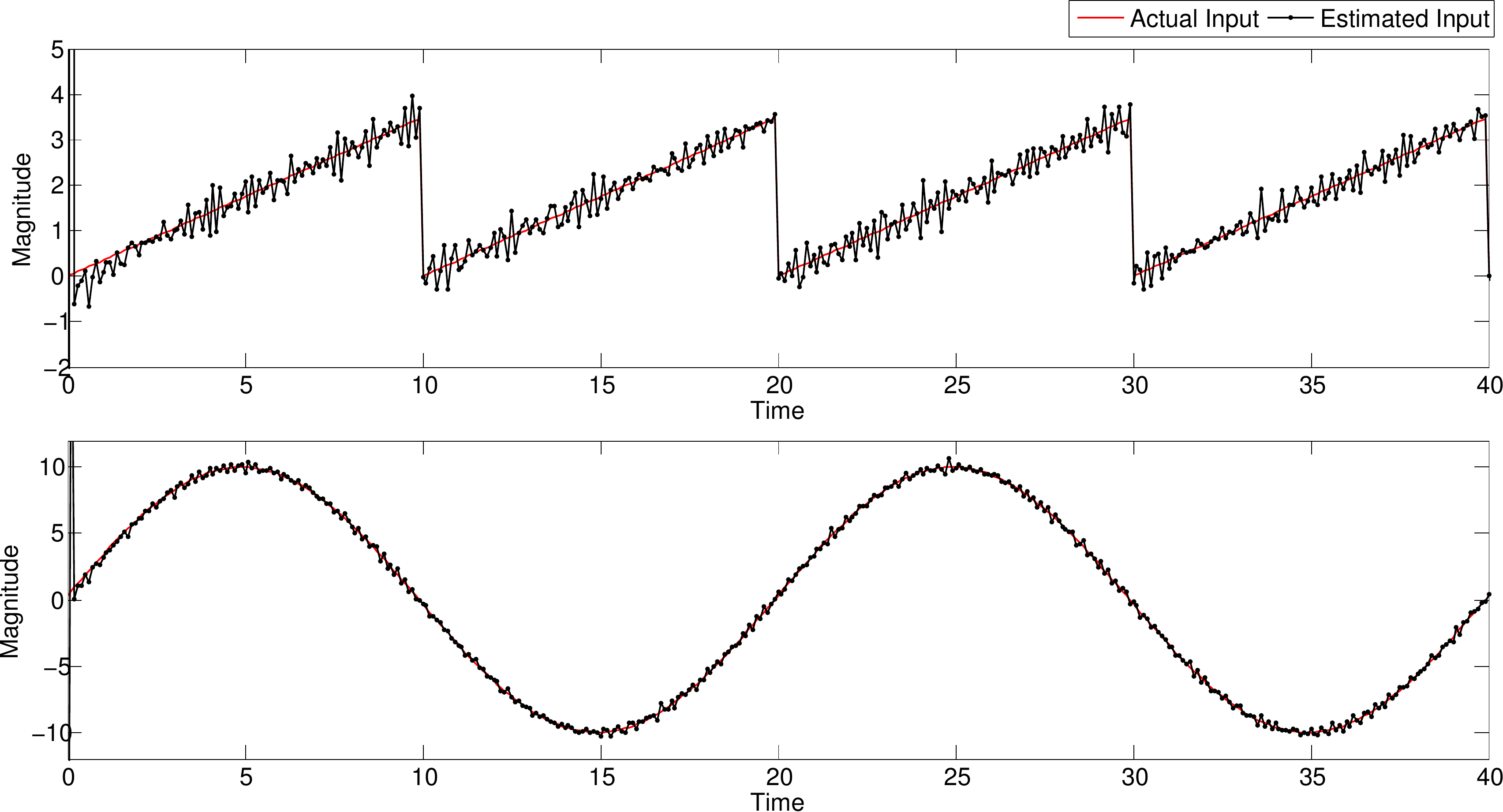}}
\caption{Actual and estimated inputs for the compartmental model example. The
filter uses a delay of one time step ($r$=1), since $CH=0$ and $CAH$ is full
rank.}
\label{fig:min1a}
\end{figure}
\begin{figure}[ht!]
\centering{
\includegraphics[width=5cm]{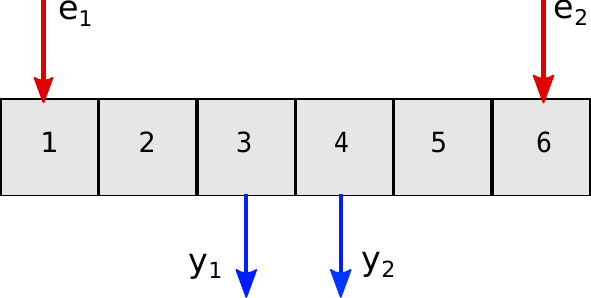}}
\caption{The inputs are provided to the first
and the last compartment while the states from the third and the fourth
compartment are measured and constitute the outputs.}
\end{figure}
The compartmental model is a convenient example to illustrate delayed input
reconstruction since by changing the compartments from which the states are
measured, the product $CA^rH$ can be altered. For instance, if the inputs are
given to the compartments 1 and 6 and the states in the compartments 3 and 4 are
measured as outputs then it follows that
\begin{align}
H = \matl{cccccc} 1 & 0 & 0 & 0 & 0 & 0\\0 & 0 & 0 & 0 & 0 & 1 \matr^\rmT,
C & =
\matl{cccccc} 0 & 0 & 1 & 0 & 0 & 0 \\
    0 & 0 & 0 & 1 & 0 & 0 \matr. \label{eq:ex_comp12}
\end{align}
Note that in this case $CH=0, CAH=0$ and $CA^2H$ is full rank. This
implies that input reconstruction is only possible using the filter
(\ref{xhat1r}) - (\ref{xhat3r}), (\ref{eq:L_squarer}), (\ref{IR_r})
with a minimum delay of two time-steps i.e. $r=2$. In this case,
results similar to those in Figure \ref{fig:min1a} are obtained, but
not shown here due to space constraints. Note that the filters in
\cite{kitanidis,steven_kitanidis,palanthACC2006} are no-delay filters
and require $CH$ to be full rank and hence cannot handle the case
where $CH=0$.
\par Next we present a numerical result
illustrating delayed input reconstruction in presence of minimum phase
zeros. Consider a state space system characterized by the following
$A,H,C$ matrices.
\begin{align}
A&= \matl{ccccc} 1.1 & -0.6 & 1\\
   0.5 & 0 & 1\\
   0 & 0.2 & 0.3\matr, 
   \quad H = \matl{ccccc} 2\\0\\0\matr,\notag \\
 C&= \matl{ccccc} 0 & 0.4 & 1\matr. \notag
\end{align}
This system has a zero at $-0.2$ (minimum phase) and the eigenvalues of
$(A-L_kCA^2)$ are $0,0,-0.2$ in accordance to Lemma \ref{lem:eig_r}. 
Figure \ref{fig:minp} shows the actual unknown input and the estimated unknown
input using the recursive delayed input reconstruction filter. 
\begin{figure}[!ht]
\centering
\includegraphics[scale=0.5]{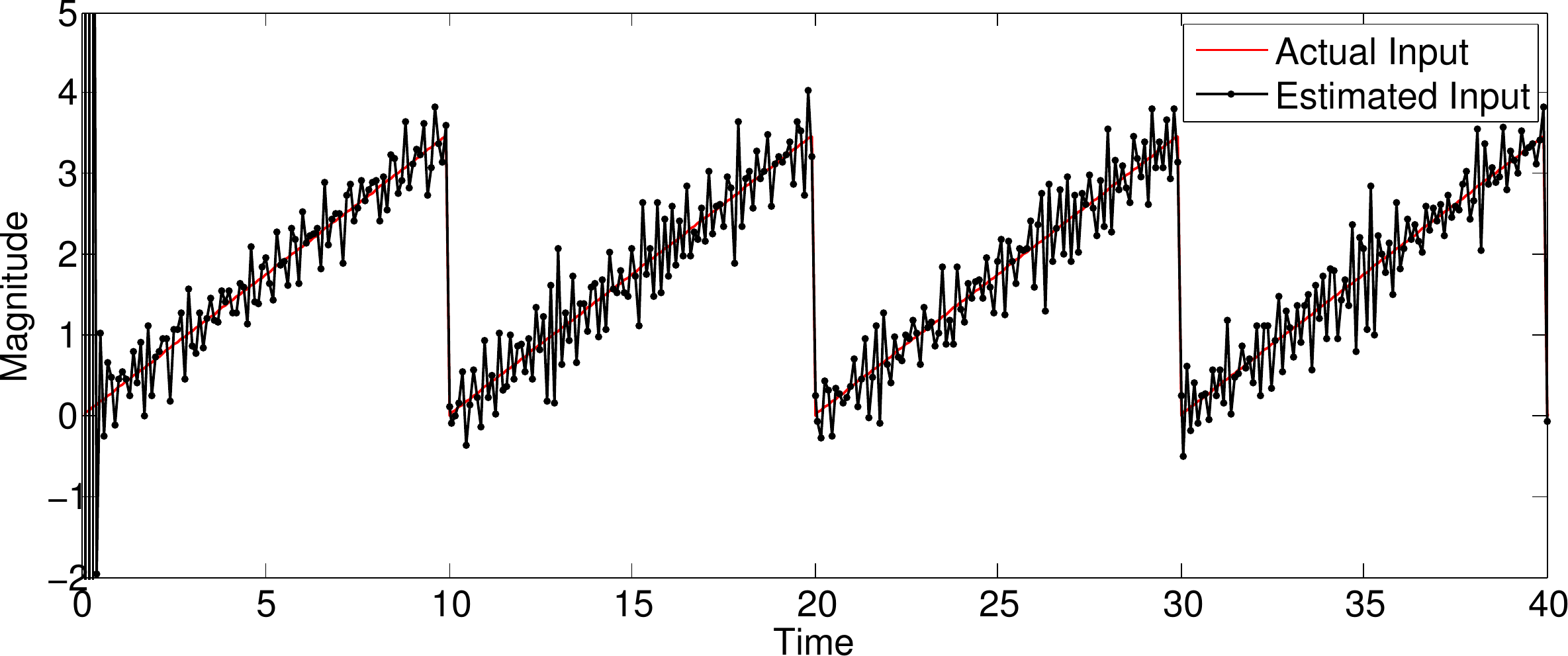}
\caption{Actual and estimated inputs for a system with minimum phase zero}
\label{fig:minp}
\end{figure}
Figure \ref{fig:ermin} shows the error in state estimation. The
computed state estimation error represents the error computed using (\ref{ers})
and substituting $r=1$ and assuming that the initial condition of the states is
known. The computation was done using the eigenvalues of $(A-L_kCA^2)$ and the
results were plotted against the actual state estimation error. The exact
overlap between the actual and the computed state estimation error consolidates
the finding that the invariant zeros of the system govern the dynamics of the
state estimation error while the decaying nature of the error is in accordance
with the Theorem \ref{thm:suff_r} with $r=1$.
\begin{figure}[!ht]
\centering
\includegraphics[width=14cm]{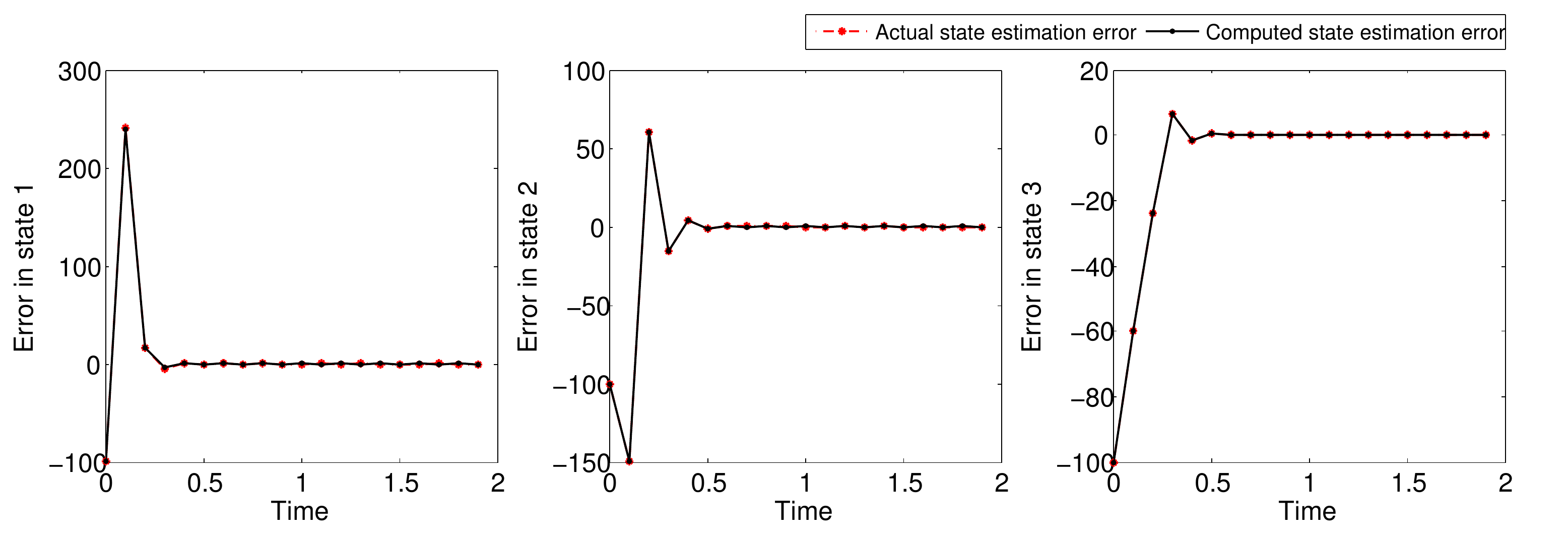}
\caption{Error in state estimation for a system with minimum phase zero}
\label{fig:ermin}
\end{figure}\\\
Now consider a state space system characterized by  the following $A,H,C$
matrices.
\begin{align}
A&= \matl{ccccc} 0.0725 & 1 & 0.2072\\
   -0.6158 & 0.0725 & 0.2339\\
   0 & 0 & -0.1449\matr, 
   \quad H = \matl{ccccc} 0\\0\\4\matr,\notag \\
 C&= \matl{ccccc} 5.005 & 0 & 0\matr. \notag
\end{align}
This system has a zero at $-1.0564$ (nonminimum phase) and the eigenvalues of
$A-L_kCA^2$ are $-1.0564,0,0$ in accordance with \ref{lem:eig_r}. 
Figure \ref{fig:nminp} shows the actual unknown input and the estimated unknown
input using the recursive delayed input reconstruction filter. 
\begin{figure}[ht!]
\centering
\includegraphics[scale=0.5]{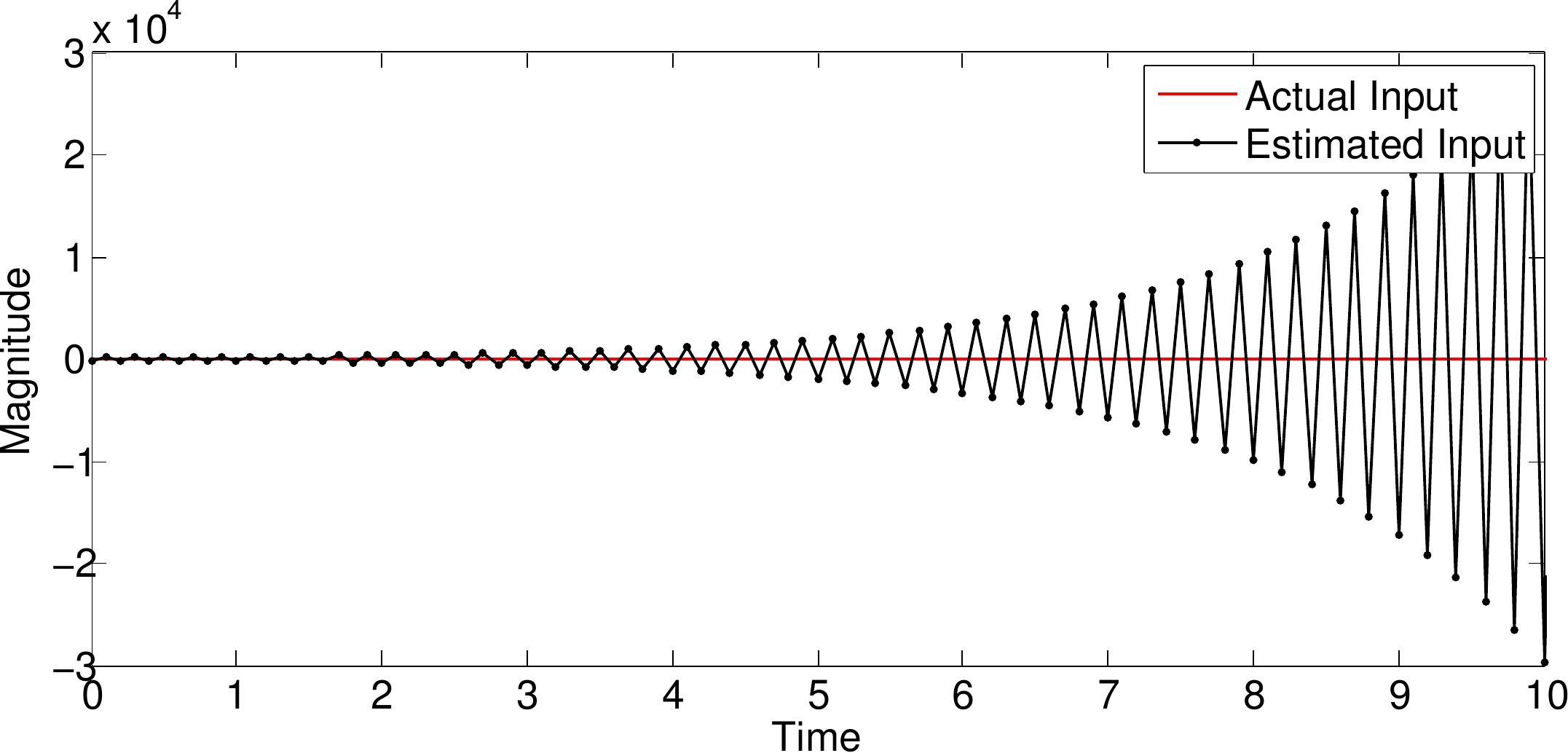}
\caption{Actual and estimated inputs for a system with nonminimum phase zero}
\label{fig:nminp}
\end{figure}
Figure \ref{fig:ernmin} shows the actual and computed error in state estimation.
\begin{figure}[ht!]
\centering
\includegraphics[width=14cm]{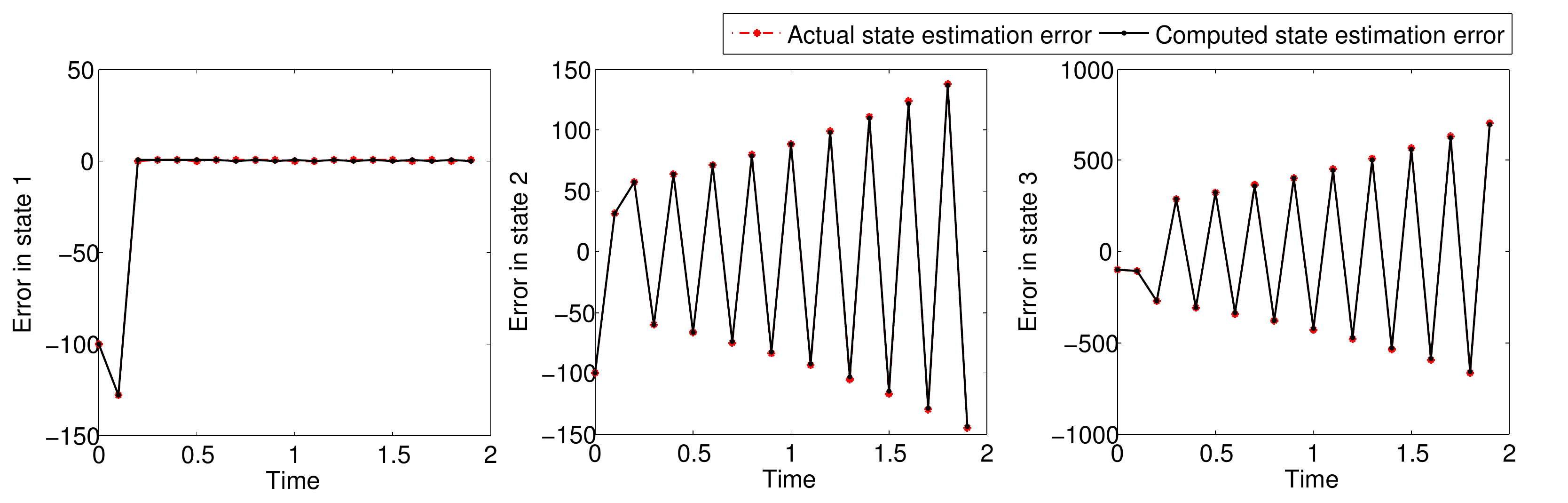}
\caption{Error in state estimation for a system with nonminimum phase zero}
\label{fig:ernmin}
\end{figure}

\subsection{Numerical Results - Non-square Systems}
State estimation and input reconstruction with a delay of one time step is
performed for the non-square state space system characterized by 
the following $A,H,C$ matrices.
\begin{align}
A&= \matl{ccccc} 0.0725 & 1 & 0.2072\\
   -0.6158 & 0.0725 & 0.2339\\
   0 & 0 & -0.1449\matr, 
   \quad H = \matl{ccccc} 0\\0\\4\matr,\notag \\
 C&= \matl{ccccc} 5.005 & 0 & 0\\0&0.1&0\matr. \notag
\end{align}. The minimum variance gain for the estimator is computed
using Corollary \ref{lcor}. Figure \ref{fig:nonsqi} shows the actual
unknown input and the estimated unknown input using the recursive
delayed input reconstruction filter. 
\begin{figure}[ht!]
\centering
\includegraphics[scale=0.5]{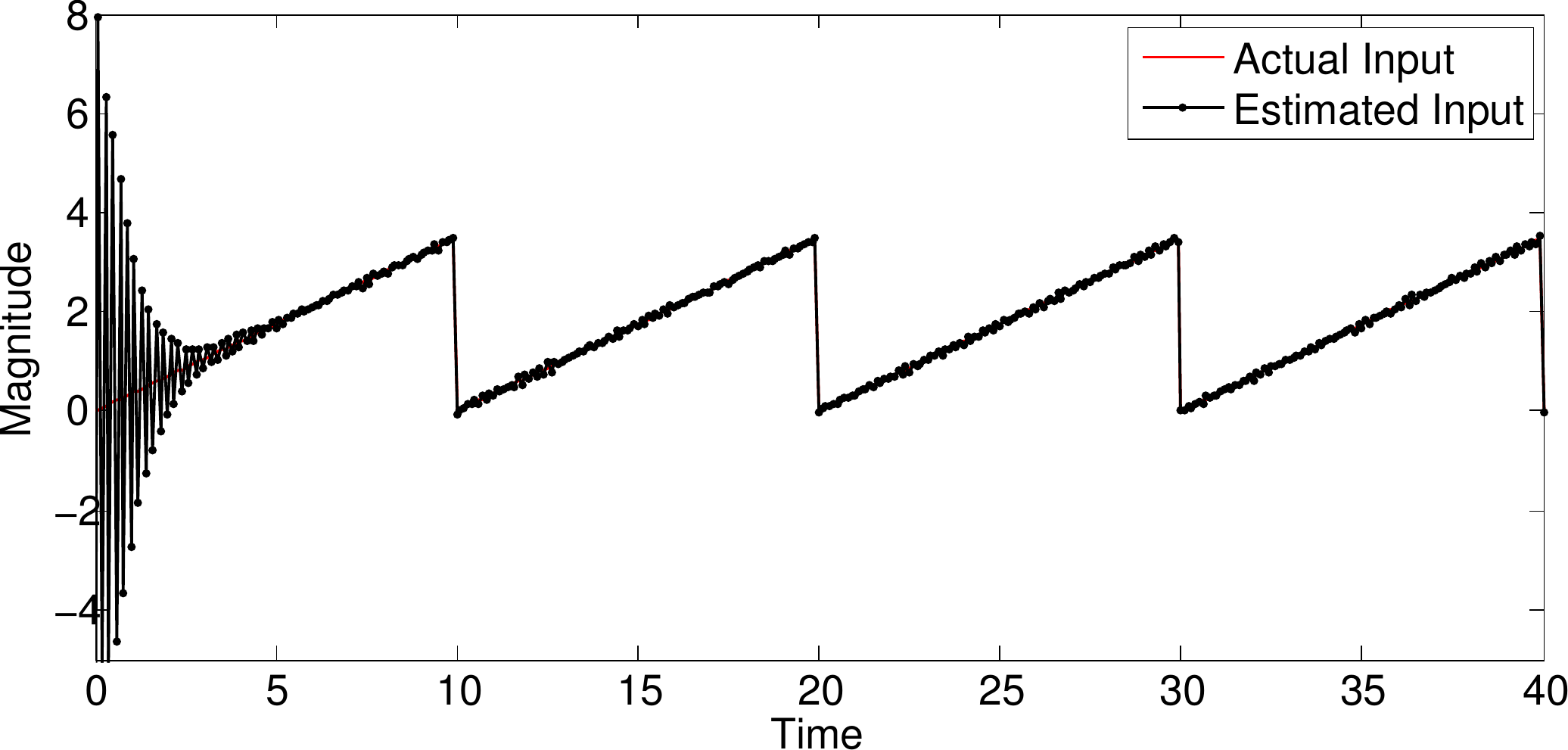}
\caption{Input reconstruction for a non-square system}
\label{fig:nonsqi}
\end{figure}
It should be noted that the convergence of the filter in this case is
asymptotic. In the next section we highlight the relationship between the
filter developed in this paper and the filters developed in
\cite{kitanidis,steven_kitanidis,palanthACC2006}.

\section{Relationship with Unbiased Minimum Variance Filters With No Delay}

In this subsection, within the context of square systems ($l=p$), we show that
the filters developed in \cite{kitanidis,steven_kitanidis,palanthACC2006} are
special cases of the filter (\ref{xhat1r}) - (\ref{xhat3r}), (\ref{eq:L_squarer}),
(\ref{IR_r})  with $r=0$.

\begin{prop} Let $l=p$, then the unbiased minimum variance filters in
\cite[(3) - (6) along with (13) and (20)]{steven_kitanidis} and \cite{kitanidis}
are equivalent to the filter (\ref{xhat1r}) - (\ref{xhat3r}),
(\ref{eq:L_squarer}), (\ref{IR_r}) with $r=0$.
\end{prop}
{\bf Proof.} The proof follows by first noting that in the case $l=p$, $CH$
($F_k$ in \cite{steven_kitanidis}) is invertible, and then following
straight-forward substitution and simplification. \qed

Note that \cite{steven_kitanidis,kitanidis} do not prove convergence of the
respective filters, therefore the above result proves convergence of the filter
through Theorem \ref{thm:suff_r} with $r=0$ by connecting the convergence with
zeros. Furthermore, note that in the case $l=p$ with $r=0$, from Theorem
\ref{thm:nec}, the filter gain must satisfy the condition
\begin{align}
LCH = H, \label{const_0}
\end{align}
and since rank $H=p$ and consequently rank $CH=p$, $L = H(CH)^{-1}$ is the only
possible $L$ that satisfies (\ref{const_0}) and hence there is a unique
solution and no concept of a minimum variance solution exists.

\section{Remarks}
While the results in this paper indicate that if the system has
nonminimum phase zeros or zeros on the unit-circle, the filter is not
convergent, it is worthwhile to note however that there may be other
approaches that may work in such cases as is being explored in
\cite{amatoACC2013}, \cite{hpmECC2014}. We also note here that for
square systems having no/minimum phase zeros is a sufficient
condition for unbiasedness/asymptotic unbiasedness. In non-square
this is not the case as the example in section \ref{non-sqr}
illustrates. Note that no-delay filters in
\cite{kitanidis,steven_kitanidis,palanthACC2006} would only be
applicable if $CH$ is full rank. So in the context of the numerical
example of a compartmental model in section \ref{comp-example}, they
would be applicable only if the the output measurements were the
states of the compartments that the unknown inputs were entering (in
this case compartments 1 and 6). Finally, in the presence of
additional known inputs $u_k$, note that the same filter equations and
theory apply with the modified equations
\begin{align}
\hat{x}_{k-r\mid k} & = \hat{x}_{k-r\mid k-1}+L_k(y_{k}-C\hat{x}_{k\mid
k-1}-Du_{k}), \\
     \hat{x}_{k-r\mid k-1} &= A\hat{x}_{k-r-1\mid k-1} + Bu_{k-r-1}, \\
     \hat{x}_{k\mid k-1} &= A^{r+1}\hat{x}_{k-r-1\mid k-1} + A^rBu_{k-r-1} +
A^{r-1}Bu_{k-r} + \cdots \nonumber \\& + Bu_{k-1} \\
\hat e_{k-r-1} & \isdef (CH)^{-1}(y_{k} - C\hat x_{k|k-1} - Du_k).
\end{align}
instead of (\ref{xhat1r}) - (\ref{xhat3r}).  It is to be noted that one drawback of the input
reconstruction method is that it cannot differentiate between unknown
noise and unknown input.

\section{Conclusions}
In this paper, we developed a technique that recursively use current
measurements to estimate past states and reconstruct past
inputs. Furthermore, we derived convergence results for the filters
developed and established its relationship with invariant zeros of the
system.  Thus we developed a broader class of filters (than
traditional unbiased minimum-variance filters), provided necessary and
sufficient conditions for the filter to provide unbiased estimates.
we also established that the unbiased-minimum variance filters in
\cite{kitanidis,steven_kitanidis,palanthACC2006,palanthUMVACC2007} are
special cases of the filter developed in this note and provided
numerical examples illustrating the key difference between the proposed
filter and existing methods. The key results are listed below
\begingroup
    \fontsize{12pt}{12pt}\selectfont
\begin{enumerate}
\item Necessary conditions for unbiasedness of the filter.
\item Sufficient conditions for convergence of the filter for square
  systems.
\item Showing that the sufficient conditions for convergence of the
  filter for square systems do not hold for non-square systems.
\item Establishing an upper bound for the filter delay and deriving existence
  conditions for the filter.
\end{enumerate} 
\endgroup
 Future work will focus on an in depth analysis of
 convergence in non-square systems and its relationship with the invariant parameters of the system.  
\footnotesize
\bibliographystyle{ieeetr}
\bibliography{DNA_refs}
\end{document}